\documentclass[12pt]{article}
\usepackage{amsmath,amssymb,amsbsy,upref}
\usepackage{amsthm}
\usepackage[T2A]{fontenc}
\usepackage[cp1251]{inputenc}
\usepackage[english]{babel}
\usepackage{bm}

\newtheorem{Theorem}{Theorem}[section]

\newtheorem{Lemma}{Lemma}[section]
\newtheorem{Proposition}{Proposition}[section]

\newtheorem{Remark}{Remark}[section]

\begin{document}
\title{Modeling of fluid flow in a flexible vessel with elastic walls}
\author{Vladimir Kozlov$^a$, Sergei Nazarov$^b$ and German Zavorokhin$^c$}
\date{}
\maketitle

\vspace{-6mm} 

\begin{center}
$^a${\it Department of Mathematics, Link\"oping University, \\ S--581 83 Link\"oping, Sweden E-mail: vlkoz@mai.liu.se\\
$^b$ St.Petersburg State University,
Universitetsky pr., 28, Peterhof, St. Petersburg, 198504, Russia, and Department of Mathematics, Link\"oping University \\

$^c$ St.Petersburg Department of the Steklov Mathematical Institute, Fontanka, 27, 191023, St.Petersburg, Russia E-mail: zavorokhin@pdmi.ras.ru}

\end{center}

\bigskip \noindent {\bf Abstract.}
We exploit a two-dimensional model \cite{4}, \cite{3} and \cite{1} describing the elastic behavior of the wall of a flexible blood vessel which takes interaction with
surrounding muscle tissue and the 3D fluid flow into account. We
study time periodic flows in a cylinder with such compound   boundary conditions. The main result is that solutions of this problem do not depend on the period and they are nothing else but the time independent Poiseuille flow. Similar solutions of the Stokes equations for the rigid wall (the no-slip  boundary condition) depend on the period  and their profile depends on time.


\vspace{-6mm} 

\section{Introduction}

In any book about the human circulatory system, one can read  that the elasticity of the composite walls of the arteries and the muscle material surrounding arteria's bed significantly contributes to the transfer of blood pushed by the heart along the arterial tree. In addition, hardening of the walls of blood vessels caused by calcification or other diseases makes it much more difficult to supply blood to peripheral parts of the system. At the same time, the authors could not find an answer to a natural question anywhere in the medical and applied mathematical literature:
 how is the elasticity support mechanism for the blood supply system? In numerous publications, modeling  the circulatory system, computational or pure theoretical, there are no fundamental differences between the steady flows of viscous fluid in pipes with rigid walls and vessels with elastic walls. Moreover, quite often much attention is ungroundedly paid to nonlinear convective terms in the Navier-Stokes equations, although blood as a multi-component viscoelastic fluid should be described by much more involved integro-differential equations. We also note that for technical reasons, none of the primary one-dimensional models of a single artery or the entire arterial tree obtained using dimension reduction procedures includes the terms generated by these nonlinear terms.

In connection with the foregoing, in this paper we consider the linear non-stationary Stokes equations simulating a heartbeat, we study time-periodic solutions in a straight infinite cylinder with an arbitrary cross-section.

In the case of the Dirichlet (no-slip) condition, according to well-known results, there are many periodic solutions $({\bf v},p)$, where the velocity ${\bf v}$ has only the component directed along the cylinder ($z$-axis) and the pressure $p$ depends linear on $z$ with coefficients depending only on the time variable $t$. For elastic walls, there is only one such solution up to a constant factor, proportional to the steady Poiseuille flow which does not depend on the time variable but can be considered as periodic one with any period. This is precisely the difference in behaviour of blood flow between elastic  and rigid walls. The former smooths (at removal from the heart) the blood flow entering the aorta, and then into the artery with sharp, clearly defined jerks, this is how the heart works with a heart valve, and the latter reproduce the frequency of the flow throughout the length of the pipe.

Due to the elastic walls of arteries, an increased heart rate only leads to an increase in the speed of blood flow (the flux grows) without changing the structure of the flow as a whole, and only at an ultra-high beat rate, when the wall elasticity is not enough to compensate for the flow pulsation, the body begins to feel heart beating. 
The fact is that the human arterial system is geometrically arranged in a very complex system, gently conical shape of blood vessels, their curvature and a considerable number of bifurcation nodes. Therefore, the considered model problem of an infinite straight cylinder gives only a basic approximation to the real circulatory system and on some of its elongated fragments, which acquires periodic disturbances found in the wrists, temples, neck and other periphery of the circulatory system. The correctness of such views is also confirmed by a full-scale experiment on watering the garden: a piston pump delivers water with shocks, but with a long soft hose the water jet at the outlet is unchanged, but with a hard short-pulsating one.

A general two-dimensional model describing the elastic behaviour of the wall of a  flexible vessel has been presented in the case of a straight vessel in \cite{3}, \cite{5}, in the case of a curved vessel in \cite{1} and for numerical results see \cite{BKKN1}, \cite{BGKN}. The wall has a laminate structure consisting of several anisotropic elastic layers of varying thickness and is assumed to be much thinner than the radius of the channel which itself is allowed to vary. This two-dimensional model takes the interaction of the wall with  surrounding or supporting elastic material of muscles and the fluid flow into account and is obtained via a dimension reduction procedure.
We study  the time-periodic  flow in the straight flexible vessel with elastic walls. In comparison with the Stokes problem for the vessel with rigid walls, we prove that 

Compared with the classical works \cite{2} and \cite{Wom2} of J.R. Womersley,  for an alternative description of the above works see \cite{VCIA}, our formulation of problem has much in common. In Wormerley's works, axisymmetric pulsative blood flow in a vessel with circular isotropic elastic wall is found as a perturbation of the steady Poisseulle flow. Apart from inessential generalizations like arbitrary shape of vessel's cross-section and orthotropic wall, the main difference of our paper is in the coefficient $K(s)$ which describe the reaction of the surrounding cell material on deformation of the wall. In other words, the vessel is assumed in \cite{2}, \cite{Wom2} to "hang in air" while in our paper it is placed inside the muscular arteria's bed as in human and animal bodies intended to compensate for external and internal influences. An evident experiment shows that a rubber or plastic hose uses to wriggle under pulsative water supply.






\section{Problem statement}
\subsection{Preliminaries}
Let $\Omega$ be a bounded, simple connected domain in the plane $\Bbb R^2$ with $C^{1,1}$ boundary $\gamma =\partial\Omega$ and let us introduce the spatial cylinder
\begin{equation}\label{shell}
{\mathcal C}=\{x=(y,z)\in \mathbb R^2\times \mathbb R:\ y=(y_1,y_2)\in\Omega,\ z\in \Bbb R\}.
\end{equation}
We assume that the curve $\gamma$ is parameterised as $(y_1,y_2)=(\zeta_1(s),\zeta_2(s))$, where $s$  is the arc length along $\gamma$ measured
counterclockwise from a certain point and $\zeta=(\zeta_1,\zeta_2)$ is a vector $C^{1,1}$ function. The length of the countour $\gamma$ is denoted by $|\gamma|$ and its curvature  by
$$
\kappa=\kappa(s)=\zeta_1''(s)\zeta_2'(s)-\zeta_2''(s)\zeta_1'(s).
$$
 In a neighborhood of
$\gamma$, we introduce the natural curvilinear
orthogonal coordinates system $(n,s)$, where $n$ is the oriented
distance to $\gamma$ ($n>0$ outside $\Omega$).

The boundary of the cylinder ${\mathcal C}$ is denoted by $\Gamma$, i.e.
\begin{equation}
\Gamma=\{x=(y,z):\,y\in\gamma, z\in\Bbb R\}.
\end{equation}

The flow in the vessel is described by the velocity vector ${\bf
v}=({\bf v}_1,{\bf v}_2,{\bf v}_3)$ and by the pressure $p$ which are subject to the non-stationary Stokes  equations:
\begin{equation}\label{Intr1}
\partial_t{\bf v}-\nu\Delta {\bf v}+\nabla p=0\;\;\mbox{and}\;\;\nabla\cdot{\bf v}=0\;\;\mbox{in ${\mathcal C}\times\Bbb R\ni (x,t)$} .
\end{equation}
Here  $\nu>0$ is the
kinematic viscosity related to the dynamic viscosity $\mu$ by
$\nu=\mu/\rho_b$, where $\rho_b>0$ is the density of the fluid.

The  elastic properties of the 2D boundary
are described by the displacement vector
${\bf u}$  defined on  $\Gamma$ and they are presented in \cite{4}, \cite{2} for a straight cylinder and in \cite{1} for a curve-linear cylinder. If we use the curve-linear coordinates $(s,z)$ on $\Gamma$ and write the vector ${\bf u}$ in the basis ${\bf n}$, ${\boldsymbol \tau}$ and ${\bf z}$, where ${\bf n}$ is the outward unit normal vector, ${\boldsymbol \tau}$ is the tangent vector to the curve $\gamma$ and ${\bf z}$ is the direction of $z$ axis, then the balance equation has  the following form:
\begin{eqnarray}\label{M1k}
&&D(\kappa,-\partial_s,-\partial_z)^T\,Q(s)
D(\kappa,\partial_s,\partial_z){\bf
u}\nonumber\\
&&+\rho(s)\partial_t^2{\bf u}+K(s){\bf u}+\sigma(s){\mathcal F}=0\;\;\mbox{in $\Gamma\times\Bbb R$},
\end{eqnarray}
where $\rho(s)$ is the average density of the vessel wall,
$\sigma=\rho_b/h$, $h=h(s)>0$ is the thickness of the wall, $A^T$ stands for the transpose of a matrix $A$, $ D(\kappa,\partial_s,\partial_z)=D_0\partial_z+D_1(\partial_s)$, where
\begin{equation}\label{Intr6j}
D_0=\left(\begin{array}{lll}
0& 0 & 0  \\
0 &0 &1\\
0 & \frac{1}{\sqrt{2}}  &0
\end{array}
\right ),\;D_1(\partial_s)=\left(\begin{array}{lll}
\kappa & \partial_s & 0  \\
0 &0 &0\\
0 & 0  &\frac{1}{\sqrt{2}} \partial_s
\end{array}
\right )
\end{equation}
and $K(s){\bf u}=(k(s){\bf u}_1,0,0)$.
Here $k(s)$ is  a scalar function,  $Q$ is a $3\times 3$ symmetric positive definite matrix of homogenized elastic moduli (see \cite{1}) and the displacement vector ${\bf u}$ is written in the curve-linear coordinates $(s,z)$ in the basis ${\bf n}$, ${\boldsymbol \tau}$ and ${\bf z}$.
Furthermore ${\mathcal F}=({\mathcal F}_n,{\mathcal F}_z,{\mathcal F}_s)$ is the hydrodynamical force  given by
\begin{equation}\label{M2}
{\mathcal F}_n=-p+2\nu\frac{\partial v_n}{\partial n},\;\;{\mathcal F}_s=\nu\Big
(\frac{\partial v_n}{\partial s}+\frac{\partial v_s}{\partial
n}-\kappa v_s\Big ),\;\;{\mathcal F}_z=\nu\Big (\frac{\partial
v_n}{\partial z}+\frac{\partial v_z}{\partial n}\Big),
\end{equation}
where $v_n$ and $v_s$ are the velocity components in the direction
of the normal ${\bf n}$ and the tangent ${\boldsymbol \tau}$, respectively,
whereas $v_z$ is the longitudinal velocity component.
The functions $\rho$, $k$ and the elements of the matrix $Q$ are bounded measurable functions satisfying
\begin{equation}\label{rok}
\rho(s)\geq\rho_0>0\;\;\mbox{ and}\;\;k(s)\geq k_0>0.
\end{equation}
 The elements of the matrix $Q$ are assumed to be Lipschitz continuous and $\langle Q\xi,\xi\rangle\geq q_0|\xi|^2$ for all $\xi\in\Bbb R^3$  with $q_0>0$, where $\langle\cdot,\cdot\rangle$ is the cartesian inner product in $\Bbb R^3$.

We note that
$$
D(\kappa,\partial_s,\partial_z){\bf u}^T=(\kappa {\bf u}_1+\partial_s{\bf u}_2,\partial_z{\bf u}_3,\frac{1}{\sqrt{2}}(\partial_z{\bf u}_2+\partial_s{\bf u}_3))^T
$$
and one can easily see that
\begin{equation}\label{Mars21a}
\kappa {\bf u}_1+\partial_s{\bf u}_2={\boldsymbol \varepsilon}_{ss}({\bf u}),;\;\partial_z{\bf u}_3={\boldsymbol \varepsilon}_{zz}({\bf u})\;\;\mbox{and}\;\;\partial_z{\bf u}_2+\partial_s{\bf u}_3=2{\boldsymbol \varepsilon}_{sz}({\bf u})
\end{equation}
on $\Gamma$. Here ${\boldsymbol \varepsilon}_{ss}({\bf u})$, ${\boldsymbol \varepsilon}_{zz}({\bf u})$ and ${\boldsymbol \varepsilon}_{sz}({\bf u})$ are components of the deformation tensor in the basis $\{{\bf n},{\boldsymbol \tau},{\bf z}\}$. In what follows we will write the displacement vector as ${\bf u}=({\bf u}_1,{\bf u}_2,{\bf u}_3)$, where ${\bf u}_1={\bf u}_n$, ${\bf u}_2={\bf u}_s$ and ${\bf u}_3={\bf u}_z$. For the velocity ${\bf v}$ we will use indexes $1,2$ and $3$ for the components of ${\bf v}$ in $y_1$, $y_2$ and $z$ directions respectively.

Furthermore the vector functions ${\bf v}$ and ${\bf u}$ are connected on the boundary by the relation
\begin{equation}\label{Feb5}
{\bf v}=\partial_t{\bf u}\;\;\mbox{on $\Gamma\times\Bbb R$.}
\end{equation}

The problem (\ref{Intr1})--(\ref{Feb5}) appears when we deal with a flow in a pipe surrounded by a thin layered elastic wall which separates the flow from the muscle tissue. Since we have in mind an application to the blood flow in the blood circulatory system, we are interested in periodic in time solutions.
One of goals of this paper is to describe all periodic solutions to the problem (\ref{Intr1}), (\ref{M1k}), (\ref{Feb5}) which are bounded in $\Bbb R\times {\mathcal C}\ni (t,x)$.

It is reasonable to compare property of solutions to this problem with similar properties of solutions to the Stokes system (\ref{Intr1}) supplied with
the no-slip boundary condition
\begin{equation}\label{Feb5d}
{\bf v}=0\;\;\mbox{on $\Gamma$.}
\end{equation}
Considering the problem (\ref{Intr1}), (\ref{Feb5d}) we  assume that the boundary $\gamma$ is Lipschitz only.

 The following result about the problem (\ref{Intr1}), (\ref{Feb5d}) is possibly known but we present a concise proof for reader's convenience.

\begin{Theorem}\label{T1} Let the boundary $\gamma$ be Lipschitz and $\Lambda>0$. There exists $\delta>0$ such that if $({\bf v},p)$ are $\Lambda$-periodic in time functions satisfying  (\ref{Intr1}), (\ref{Feb5d}) and may admit a certain exponential growth at infinity
\begin{equation}\label{Feb17aa}
\max_{0\leq t\leq\Lambda}\int_{\mathcal C}e^{-2\delta |z|}(|\nabla {\bf v}(x,t)|^2+|\nabla\partial_z{\bf v}|^2+|p(x,t)|^2)dx <\infty.
\end{equation}
 Then
\begin{equation}\label{Feb17ba}
p=zp_*(t)+p_0(t),\;\;{\bf v}(x,t)=(0,0,{\bf v}_3(y,t)),
\end{equation}
where $p$ and ${\bf v}_3$ are $\Lambda$-periodic functions in $t$ which satisfy the problem
\begin{eqnarray}\label{Feb17ca}
&&\partial_t{\bf v}_3-\nu\Delta_y {\bf v}_3+p_*(t)=0\;\;\mbox{in $\Omega\times\Bbb R$}\nonumber\\
&&{\bf v}=0\;\;\mbox{on $\gamma\times\Bbb R$}.
\end{eqnarray}
\end{Theorem}
 Thus the dimension of the space of periodic solutions to the problem (\ref{Intr1}), (\ref{Feb5d}) is infinite and they can be parameterised by a periodic function $p_*$.
In the case of elastic wall situation is quite different

\begin{Theorem}\label{T2} Let the boundary $\gamma$ be $C^{1,1}$ and $\Lambda>0$. Let also $({\bf v},p,{\bf u})$ be a $\Lambda$-periodic with respect to $t$ solution to the problem
(\ref{Intr1})--(\ref{Feb5}) admitting an arbitrary power growth at infifnity
\begin{eqnarray}\label{Mars5}
&&\max_{0\leq t\leq\Lambda}\Big(\int_{\mathcal C}(1+|z|)^{-N} (|{\bf v}|^2+|\nabla_x {\bf v}|^2+|\nabla_x\partial_z {\bf v}|^2+|p|^2)dx\\
&&+\int_{\Gamma}(1+|z|)^{-N}(|u|^2+\sum_{k=2}^3(|\nabla_{sz}u_k|^2+|\nabla_{sz}\partial_zu_k|^2))dsdz\Big)< \infty\nonumber
\end{eqnarray}
for a certain $N>0$.
Then
\begin{equation}\label{Mars1}
p=zp_0+p_1,\;\;{\bf v}_1= {\bf v}_2=0,\;\;{\bf v}_3=p_0{\bf v}_*(y),
\end{equation}
where $p_0$ and $p_1$ are constants and ${\bf v}_*$ is the Poiseuille profile, i.e.
\begin{equation}\label{M9a}
\nu\Delta_y{\bf v}_*=1\;\;\;\mbox{in $\Omega$\;\; and}\;\;{\bf v}_*=0\;\;\mbox{on $\gamma$}.
\end{equation}
The boundary displacement vector ${\bf u}={\bf u}(s,z)$ satisfies the equation
\begin{equation}\label{Mars1a1}
D(\kappa(s),-\partial_s,-\partial_z)^T\,Q(s)
D(\kappa(s),\partial_s,\partial_z){\bf u}+K{\bf u}=\sigma(p,0,p_0\nu \partial_n{\bf v}_3|\gamma)^T.
\end{equation}
If the elements $Q_{21}$ and $Q_{31}$ vanish then the function ${\bf u}$ is a polynomial of second degree in $z$: ${\bf u}(s,z)=(0,\alpha,\beta)^Tz^2+{\bf u}^{(1)}(s)z+{\bf u}^{(2)}(s)$, where $\alpha$ and $\beta$ are constants.
\end{Theorem}

Thus, in the case of elastic wall all periodic solutions  are independent of $t$ and hence are the same for any period. Moreover inside the cylinder the flow takes
the Poiseuille form.
The above theorems have different requirements on the behavior of solutions with respect to $z$, compare (\ref{Feb17aa}) and (\ref{Mars5}). This is because of the following reason. In the case of the Dirichlet boundary condition we can prove a resolvent estimate on the imaginary axis ($\lambda=i\omega$, $\omega$ is real) with exponential weights independent on $\omega$. In the case of the elastic boundary condition exponential weights depends on $\omega$. Becuase of that we can not put in (\ref{Mars5}) the same exponential weight as in (\ref{Feb17aa}).

The structure of our paper is the following. In Sect.\ref{Sect1} we treat the Stokes system with the no-slip condition on the boundary of cylinder. Since we are dealing with time-periodic solutions the problem can be reduced to a series of time independent problems with a parameter (frequency). The main result there is Theorem \ref{T4}. Using this assertion it is quite straightforward to proof the main theorem \ref{T1} for the Dirichlet problem. Parameter dependent problems are studied in Sect.\ref{SEC1a}-\ref{SEC2a}. Theorem \ref{T4} is proved in Sect.\ref{S35}.

Stokes problem in a vessel with elastic walls is considered in Sect.\ref{Sect2x}. We also reduce the time periodic problem to a series of time independent problems depending on a parameter. The main result there is Theorem \ref{T26a}. Using this result we prove our main theorem \ref{T2} for the case of elastic wall in Sect.\ref{SEC3a}. The parameter depending problem is studied in Sect.\ref{Sub1}-\ref{Sub5}. The proof of  Theorem \ref{T26a} is given in Sect.\ref{Sect27a}. In Sect.\ref{Sub6} we consider the case when the parameter in the elastic wall problem is vanishing. This consideration completes the proof of Theorem \ref{T2}.

\section{Dirichlet problem for the Stokes system}\label{Sect1}

The first step in the proof of Theorem \ref{T1} is the following reduction of the time dependent problem to time independent one. Due to $\Lambda$-periodicity of our solution we can represent it in the form
\begin{equation}\label{Mars20a}
{\bf v}(x,t)=\sum_{k=-\infty}^\infty {\bf V}_k(x)e^{2\pi kit/\Lambda},\;\;p(x,t)=\sum_{k=-\infty}^\infty {\bf P}_k(x)e^{2\pi kit/\Lambda},
\end{equation}
where
\begin{equation}\label{Mars20aa}
{\bf V}_k(x)=\frac{1}{\Lambda}\int_0^\Lambda {\bf v}(x,t)e^{-2\pi kit/\Lambda}dt,\;\;P_k(x)=\frac{1}{\Lambda}\int_0^\Lambda p(x,t)e^{-2\pi kit/\Lambda}dt.
\end{equation}
These coefficients satisfy the following time independent problem

\begin{equation}\label{Feb5ba}
i\omega{\bf V}-\nu\Delta {\bf V}+\nabla
P={\bf F}\;\;\mbox{and}\;\;\nabla\cdot{\bf V}=0\;\;\mbox{in ${\mathcal C}$},
\end{equation}
 with the Dirichlet boundary condition
\begin{equation}\label{Feb5da}
{\bf V}=0\;\;\mbox{on $\Gamma$}
\end{equation}
and with $\omega=2\pi k/\Lambda$ and ${\bf F}=({\bf F}_1,{\bf F}_2,{\bf F}_3)=0$ (for further analysis it is convenient to have an arbitrary ${\bf F}$).


\begin{Theorem}\label{T4} There exist a positive number $\beta^*$ depending only on $\Omega$ and $\nu$ such that
for $\beta\in (0,\beta^*)$ the only solution to problem (\ref{Feb5ba}), (\ref{Feb5da}) with ${\bf F}=0$ which may admit a certain exponential growth
\begin{equation}\label{Feb17a}
\int_{\mathcal C}e^{-2\beta |z|}\big(|\nabla {\bf V}|^2+|\nabla \partial_z{\bf V}|^2+|P|^2\big)dydz<\infty
\end{equation}
is
\begin{equation}\label{Feb17b}
{\bf V}(x)=p_0(0,0,{\hat v}(y))\;\;\mbox{and}\;\;P(x)=p_0z+p_1,
\end{equation}
where $p_0$ and $p_1$ are constants and ${\hat v}$ satisfies
\begin{equation}\label{Feb17c}
i\omega \hat{v}-\nu\Delta \hat{v}+1=0\;\;\mbox{in $\Omega$},\;\;\hat{v}=0\;\;\mbox{on $\gamma$.}
\end{equation}
\end{Theorem}

\begin{Remark}\label{R1} From (\ref{Feb17c}) it follows that
$$
\int_\Omega \hat{v}(y)dy=i\omega\int_\Omega |\hat{v}|^2du-\nu\int_\Omega |\nabla\hat{v}|^2dy,
$$
i.e. the flux does not vanish for this solution.
\end{Remark}

We postpone the proof of the above theorem to Sect.\ref{S35} and in the next section we present the proof of Theorem \ref{T1}

\subsection{Proof of Theorem \ref{T1}}

By (\ref{Feb17aa})
\begin{equation}\label{Feb17aaa}
\int_{\mathcal C}e^{-2\delta |z|}(|\nabla {\bf V}_k(x)|^2+|{\bf P}_k(x)|^2)dx <\infty,
\end{equation}

Applying Theorem \ref{T4} and assuming $\delta<\beta^*$, we get that
$$
V_k=(0,0,\hat{v}_k(y)),\;\;P_k=zp_{0k}+p_{1k},
$$
where $p_{0k}$ and $p_{1k}$ are constants. This implies that ${\bf v}_1={\bf v}_2=0$, ${\bf v}_3$ depends only on $y, t$ and $p=p_0(t)z+p_1(t)$, which proves the required assertion.

\subsection{System for coefficients (\ref{Feb5ba}), (\ref{Feb5da})}\label{SEC1a}

To describe the main solvability result for the problem (\ref{Feb5ba}), (\ref{Feb5da}), let us introduce some function spaces. For $\beta\in\Bbb R$ we denote by $L^{2}_\beta({\mathcal C})$ the space of functions on ${\mathcal C}$  with the finite norm
$$
||u;L^{2}_\beta({\mathcal C})||=\Big(\int_{{\mathcal C}}e^{2\beta z}|u(x)|^2dx\Big)^{1/2}.
$$
By $W^{1,2}_\beta({\mathcal C})$ we denote the space of functions in ${\mathcal C}$ with the finite norm
$$
||v;W^{1,2}_\beta({\mathcal C})||=\int_{\mathcal C}e^{2\beta z}(|\nabla_x v|^2+|v|^2)dx\Big)^{1/2}.
$$
 We will use the same notation for spaces of vector functions.

\begin{Proposition}\label{Prop1} Let the boundary $\gamma$ is Lipschitz and $\omega\in\Bbb R$. There exist $\beta^*>0$ independent of $\omega$ such that the following assertions are valid:

 {\rm (i)} for any $\beta\in (-\beta^*,\beta^*)$, $\beta\neq 0$ and  ${\bf F}\in L^2_\beta({\mathcal C})$,  the problem (\ref{Feb5ba}), (\ref{Feb5da}) has a unique solution $({\bf V},{\bf P})$ in $W_\beta^{1,2}({\mathcal C})^3\times L^2_\beta({\mathcal C})$ satisfying the estimate
\begin{equation}\label{Mars30a}
||{\bf V};W^{1,2}_\beta({\mathcal C})||+||{\bf P};L_2^\beta({\mathcal C})||\leq C||{\bf F};L_2^\beta({\mathcal C})||.
\end{equation}
where $C$ may depend on $\beta$, $\nu$ and $\Omega$. Moreover,
\begin{equation}\label{Mars30av}
||\partial_z{\bf V};W^{1,2}_\beta({\mathcal C})||\leq C||{\bf F};L^2_\beta({\mathcal C})||.
\end{equation}

{\rm (ii)} Let $\beta\in (0,\beta^*)$ and ${\bf F}\in L^2_\beta({\mathcal C})\cap L^{2}_{-\beta}({\mathcal C})$. Then solutions $({\bf V}_\pm,{\bf P}_\pm)\in W_{\pm\beta}^{1,2}({\mathcal C})^3\times L^{2}_{\pm\beta}({\mathcal C})$ to (\ref{Feb5ba}), (\ref{Feb5da}) are connected by
\begin{equation}\label{Mars27a}
{\bf V}_-(x)={\bf V}_+(x)+p_0(0,0,\hat{v}(y)),\;\;{\bf P}_-(x)={\bf P}_+(x)+p_0z+p_1,
\end{equation}
with certain constants $p_0$ and $p_1$. Here $\hat{v}$ is solution to (\ref{Feb17c}).

\end{Proposition}

\begin{Remark} If $\gamma$ is $C^{1,1}$ then  it follows from \cite{Fa2} that the left-hand side in (\ref{Mars30a}) can be replaced by
$$
||{\bf V};W^{2,2}_\beta({\mathcal C})||+||\nabla {\bf P};L_2^\beta({\mathcal C})||.
$$

\end{Remark}

\subsection{Operator pencil, weak formulation}\label{S19a}

We will use the spaces of complex valued functions $H^1_0(\Omega)$, $L^2(\Omega)$ and $H^{-1}(\Omega)$ and the corresponding norms are denoted by $||\cdot ||_1$, $||\cdot ||_0$ and $||\cdot ||_{-1}$ respectively.

Let us introduce an operator pencil by
\begin{equation}\label{Feb7a}
{\mathcal S}(\lambda)\begin{pmatrix} v \\ p \end{pmatrix}=\begin{pmatrix} \mu  v_1-\nu\Delta_y v_1+\partial_{y_1}p \\
\mu v_2-\nu\Delta_y v_2+\partial_{y_2}p\\
\mu v_3-\nu\Delta_y v_3+\lambda p\\
\partial_{y_1}v_1+\partial_{y_2}v_2+\lambda v_3\end{pmatrix},\;\;\mu=i\omega -\nu\lambda^2,
\end{equation}
where $v=(v_1,v_2,v_3)$ is a vector function  and $p$ is a scalar function in $\Omega$ . This pencil is defined for vectors $(v,p)$  such that $v=0$ on $\gamma$.

Clearly
\begin{equation}\label{Feb20a}
{\mathcal S}(\lambda)\; :\;H^1_0(\Omega)^3\times L^2(\Omega)\rightarrow (H^{-1}(\Omega))^3\times L^2(\Omega)
\end{equation}
is a bounded operator for all $\lambda\in\Bbb C$. The following problem is associated with this operator
\begin{eqnarray}\label{Feb7b}
&&\mu  v_1-\nu\Delta_y v_1+\partial_{y_1}p=f_1,\nonumber\\
&&\mu v_2-\nu\Delta_y v_2+\partial_{y_2}p=f_2,\nonumber\\
&&\mu v_3-\nu\Delta_y v_3+\lambda p=f_3\;\;\mbox{in $\Omega$}
\end{eqnarray}
and
\begin{equation}\label{Feb7c}
\partial_{y_1}v_1+\partial_{y_2}v_2+\lambda v_3=h\;\;\mbox{in $\Omega$}
\end{equation}
supplied with the Dirichlet condition
\begin{equation}\label{Feb7d}
v=0\;\;\mbox{on $\partial\Omega$.}
\end{equation}


The corresponding sesquilinear form is given by
\begin{eqnarray*}
&&\Bbb A(v,p;\hat{v},\hat{p};\lambda)=\int_\Omega \sum_{j=1}^3(\mu v_j\overline{\hat{v}_j}+\nu \nabla_yv_j\cdot\nabla_y\overline{\hat{v}_j})dy\\
&&-\int_\Omega p\overline{(\nabla_y\cdot\hat{v'}-\bar{\lambda}\hat{v}_3)}dy+\int_\Omega(\nabla_y\cdot v'+\lambda v_3)\overline{\hat{p}}dy
\end{eqnarray*}
where $v'=(v_1,v_2)$. This form is well-defined on $H_0^1(\Omega)^3\times L^2(\Omega)$.

The weak formulation of (\ref{Feb7b})--(\ref{Feb7d}) can be written as
\begin{equation}\label{Mars10b}
\Bbb A(v,p;\hat{v},\hat{p};\lambda)=\int_\Omega (f\cdot\overline{\hat{v}}+h\overline{\hat{p}})dy
\end{equation}
for all $(\hat{v},\hat{p})\in H_0^1(\Omega)^3\times L^2(\Omega)$.
As it was shown in the proof of  Lemma 3.2(ii)\cite{Fa} the operator ${\mathcal S}(\lambda)$ is isomorphism for $\lambda=i\xi$, $\xi\in\Bbb R\setminus \{0\}$.
Since the operator corresponding to the difference of the forms $\Bbb A$ for different $\lambda$ is compact, the operator pencil
${\mathcal S}(\lambda)$ is  Fredholm for all $\lambda\in\Bbb C$ and its spectrum   consists of isolated eigenvalues of finite algebraic multiplicity, see \cite{GK}.


\subsection{Operator pencil near the imaginary axis $\Re\lambda=0$}\label{SEC2a}

Here we consider the right-hand sides in (\ref{Feb7b})--(\ref{Feb7d}) as follows $f\in L^2(\Omega)$ and $g\in L^2(\Omega)$

The next assertion is  proved in Lemma 3.2(i) \cite{Fa}, after a straightforward modification.

\begin{Lemma}\label{L10a} Let $h\in L^2(\Omega)$ and $\lambda\in\Bbb C$, $\lambda\neq 0$. Then the equation
\begin{equation}\label{Mars6}
\partial_{y_1}w_1+\partial_{y_2}w_2+\lambda w_3=h\;\;\mbox{in $\Omega$}
\end{equation}
has a solution $w\in  H_0^1(\Omega)$ satisfying the estimate
$$
\sum_{j=1}^2||w_j||_{1}\leq C(||h||_0+|\alpha|),\;\;||w_3||_1\leq C\frac{|\alpha|}{|\lambda|}
$$
where
$$
\alpha=\int_\Omega hdy
$$
and $C$ depends only on $\Omega$. The mapping $h\mapsto w$ can be chosen linear.
\end{Lemma}

The proof of the next lemma can be extracted from the proof of Lemma 3.2(ii) \cite{Fa}.

\begin{Lemma}\label{LEk1} Let $f\in L_2(\Omega)$, $h=0$ in (\ref{Feb7c}) and  let $\lambda=i\xi$, $0\neq \xi\in\Bbb R$. Then the solution to {\rm (\ref{Feb7b})--(\ref{Feb7d})} admits the estimate
\begin{equation}\label{Feb20b}
(1+|\omega|+|\xi|^2)||v||_0+(1+|\omega|+|\xi|^2)^{1/2}||\nabla v||_0\leq C||f||_0,
\end{equation}
and
\begin{equation}\label{Feb20bw}
 ||p||_0\leq C\frac{1+|\lambda|}{|\lambda|}||f||_0,
\end{equation}
where the constant $C$ depends only on $\nu$ and $\Omega$.
\end{Lemma}
\begin{proof} To estimate  the norm of $v$ by the right-hand side in (\ref{Feb20b}), we take $\hat{v}=v$ in (\ref{Mars10b}) and obtain
$$
\int_\Omega (\mu |v|^2+\nu |\nabla_yv|^2)dy=\int_\Omega f\cdot\bar{v}dy.
$$
This implies the inequality (\ref{Feb20b}).
 To estimate the norm of $p$, we choose $\nabla_{i\xi} w=p$. Then the relation (\ref{Mars10b}) becomes
$$
\int_\Omega |p|^2dy=\int_\Omega (\mu v\cdot\overline{w}+\nu\nabla v\cdot\nabla\overline{w}-f\overline{w})dy.
$$
Since by Lemma \ref{L10a}
$$
||w||_1\leq C\frac{1+|\lambda|}{|\lambda|}||p||_0,
$$
we arrive at (\ref{Feb20bw}).
\end{proof}

\begin{Lemma}\label{LEk2} There exists a positive number $\delta$  depending on $\nu$ and $\Omega$ such that if $f\in L_2(\Omega)$, $h=0$  and  $0<|\lambda|<\delta$, then the problem
{\rm (\ref{Feb7b})--(\ref{Feb7d})} has a unique solution satisfying
\begin{equation}\label{Feb20ba}
(|\omega|+1)||v||_0+(|\omega|+1)^{1/2}||v||_1+ ||p-p_m||_0+|\lambda|\,||p||_0\leq C||f||_0,
\end{equation}
where the constant $C$ depends only on $\nu$ and $\Omega$ and
$$
p_m=\frac{1}{|\Omega|}\int_\Omega pdy.
$$
\end{Lemma}

\begin{proof} From (\ref{Feb7c})) it follows that
\begin{equation}\label{Mars11ab}
\int_\Omega v_3dy=0\;\;\mbox{ when}\;\; \lambda\neq 0.
\end{equation}
Taking $\hat{v}=v$ in (\ref{Mars10b}) yields
$$
\mu\int_\Omega |v|^2dy+\nu\int_\Omega |\nabla_yv|^2dy=-2\Re\Big(\lambda\int_\Omega p\overline{v_3}dy\Big)+\Re\Big(\int_\Omega f\cdot\overline{v}dy\Big).
$$
For a small $\lambda\neq 0$, this together with (\ref{Mars11ab}) implies
$$
(1+|\omega|)||v||_0\leq C(||f||_0+|\lambda ||p-p_m||_0)
$$
and
$$
(1+|\omega|)^{1/2}||\nabla v||_0\leq C(||f||_0+|\lambda ||p-p_m||_0)
$$
with $C$ independent of $\omega$ and $\lambda$.
Taking $\hat{v}=w=(w_1,w_2,0)$ in (\ref{Mars10b}) where $w_k\in H^1_0(\Omega)$ satisfies
$$
\partial_{y_1}w_1+\partial_{y_2}w_2=p-p_m,\;\;||w||_1\leq C||p-p_m||_0,
$$
we get
$$
||p-p_m||_0^2=\mu\int_\Omega v\cdot\overline{w}dy+\nu\int_\Omega \nabla_yv\cdot\nabla_y\bar{w}dy.
$$
Therefore,
$$
||p-p_m||_0^2\leq C\big( (1+|\omega|)||v||_0+||\nabla v||_0)||p-p_m||_0
$$
and, hence,
$$
||p-p_m||_0\leq C(||f||_0+|\lambda ||p-p_m||_0)
$$
which implies
\begin{equation}\label{Apr1a}
||p-p_m||_0\leq C||f||_0\;\;\; \mbox{and}\;\; (1+|\omega|)||v||_0+(1+|\omega|)^{1/2}||\nabla v||_0\leq C||f||_0.
\end{equation}
Now taking $\hat{v}=w=(w_1,w_2,w_3)$ in (\ref{Mars10b}), where $w_k\in H^1_0(\Omega)$ is subject to
$$
\partial_{y_1}w_1+\partial_{y_2}w_2-\bar{\lambda}w_3=p-p_m,\;\;||w||_1\leq C\frac{1+|\lambda|}{|\lambda|}||p-p_m||_0,
$$
we obtain
$$
||p||_0\leq C\frac{1+|\lambda|}{|\lambda|}||f||_0.
$$
The last inequality together with (\ref{Apr1a}) gives (\ref{Feb20ba}).
\end{proof}

Now we can describe properties of the pencil ${\mathcal S}$ in a neighborhood of the imaginary axis $\Re\lambda=0$.


\begin{Lemma}\label{Lem1a} There exist $\beta^*>0$ such that the following assertion are valid:

{\rm (i)} the only eigenvalue of ${\mathcal S}$ in the strip $|\Re\lambda |<\beta^*$ is zero;

{\rm (ii)}  if $\beta\in (-\beta^*,0)\cup (0,\beta^*)$ then for $f\in L^2(\Omega)$, $h=0$ and $\Re\lambda=\beta$ the problem (\ref{Feb7b})--(\ref{Feb7d}) has a unique solution in $H^1_0(\Omega)^3\times L^2(\gamma)$ and its norm is estimated as follows:
\begin{equation}\label{Mars10a}
(1+|\lambda|^2+|\omega|)\,||v||_0+(1+|\lambda|^2+|\omega|)^{1/2}||\nabla v||_0+||p||_0\leq C||f||_0,
\end{equation}
 where the constant $C$ may depend on $\beta$, $\nu$ and $\Omega$.

\end{Lemma}
\begin{proof} First, we observe that
$$
\Bbb A(v,p;\hat{v},\hat{p};\beta+i\xi)-\Bbb A(v,p;\hat{v},\hat{p};\xi)=\nu\beta(\beta+2\xi i)\int_\Omega v\cdot\overline{\hat{v}}dy+\beta\int_\Omega (v_3\overline{\hat{p}}-p\overline{\hat{v_3}})dy.
$$
 Thus the first form is a small perturbation of the second one. Now using Lemmas \ref{LEk1} and \ref{LEk2}  for small $\beta$ we arrive at the existence of $\beta^*$, which satisfies (i). Moreover, the estimates (\ref{Feb20b}) and (\ref{Feb20bw}) are true for $\lambda=\beta+i\xi$ for a fixed $\beta\in (-\beta^*,0)\cup (0,\beta^*)$ and arbitrary real $\xi$. With this the constants in (\ref{Feb20b}) and (\ref{Feb20bw}) may depend now on $\beta$, $\nu$ and $\Omega$ only. 

\end{proof}

\subsection{Proof of Proposition \ref{Prop1} and Theorem \ref{T4}}\label{S35}

\begin{proof}
 The assertion (i) in Proposition \ref{Prop1} is obtained from Lemma \ref{Lem1a}(ii)  by using the inverse Fourier transform together with Parseval's identity.

To prove (ii) in Proposition \ref{Prop1}, we observe that
$$
({\bf V}_\pm,{\bf P}_\pm)=\frac{1}{2\pi i}\int_{\Re\lambda=\pm\beta}e^{-\lambda z}{\mathcal S}^{-1}(f(y,\lambda),0)d\lambda
$$
and the relation (\ref{Mars27a}) is obtained by applying the residue theorem.
\end{proof}

\bigskip
Now we turn to
{\bf Proof of Theorem \ref{T4}}. Let $({\bf V},{\bf P})$ be a solution to (\ref{Feb5ba}), (\ref{Feb5da}) satisfying (\ref{Feb17a}). Our first step is to construct a  representation of the solution in the form
\begin{equation}\label{Apr31ab}
({\bf V},{\bf P})=({\bf V}^{(+)},{\bf P}^{(+)})+({\bf V}^{(-)},{\bf P}^{(-)}),
\end{equation}
where ${\bf V}^{(\pm)}\in W^{1,2}_{\pm\beta}({\mathcal C})$, ${\bf P}^{(\pm)}\in L^2_{\pm\beta}({\mathcal C})$ and they solve the problem (\ref{Feb5ba}), (\ref{Feb5da}) with certain ${\bf F}={\bf F}^{(\pm)}\in L^2_{\pm\beta}({\mathcal C})$.

By the second equation in (\ref{Feb5ba}) and by (\ref{Feb5da}) the flux
\begin{equation}\label{Apr31a}
\Psi=\int_\Omega {\bf V}_3(y,z)dy\;\;\mbox{is constant}.
\end{equation}
 The vector-function $(0,0,p_0\hat{v}_3,p_0)$ with a constant $p_0$ verifies  the homogeneous problem (\ref{Feb5ba}), (\ref{Feb5da}) and its flux  does not vanish in the case $p_0\neq 0$. So subtracting it with appropriate constant $p_0$ from $({\bf V},{\bf P})$ we can reduce the proof of theorem to the case $\Psi=0$. In this way we assume in what follows that this is the case.

Let $\zeta(z)$ be a smooth cut-off function equal $1$ for large positive $z$ and $0$ for large negative $z$ and let $\zeta'$ be its derivative. We choose in (\ref{Apr31ab})
$$
({\bf V}^{(+)},{\bf P}^{(+)})=\zeta ({\bf V},{\bf P})+\zeta'({\bf W},0),\;\;({\bf V}^{(-)},{\bf P}^{(-)})=(1-\zeta )({\bf V},{\bf P})-\zeta'({\bf W},0)
$$
where the vector function ${\bf W}=({\bf W}_1,{\bf W}_2,0)$ is such that
\begin{equation}\label{Apr31b}
\nabla_y\cdot({\bf W}_1,{\bf W}_2)(y,z)={\bf V}_3(y,z).
\end{equation}
We construct solution ${\bf W}$ by solving two-dimensional Stokes problem in $\Omega$ depending on the parameter $z$:
$$
-\nu\Delta {\bf W}_k+\partial_{y_k}{\bf Q}=0,\;\;k=1,2,\;\;\partial_{y_1}{\bf W}_1(y)+\partial_{y_2}{\bf W}_2(y)=V_3(y,z)\;\;\mbox{in $\Omega$}
$$
and ${\bf W}_k=0$ on $\gamma$, $k=1,2$. This problem has a solution in $H^1_0(\Omega)^2\times L^2(\Omega)$, which is unique if we require $\int_\Omega {\bf Q}dy=0$. If we look on the dependence on the parameter $z$ it is the same as in the right-hand side. So
$$
{\bf W}_k,\partial_z{\bf W}_k,\partial^2_z{\bf W}_k\in L^2_{\rm loc}(\Bbb R;H^1_0(\Omega))\;\;\mbox{and}\;\;
{\bf Q},\partial_z{\bf Q}\in L^2(\Omega).
$$
Therefore,
$$
i\omega {\bf V}_k^{(+)}-\nu\Delta {\bf V}_k^{(+)}+\partial_{y_k}{\bf P}^{(+)}={\bf F}_k^{(+)}
$$
where
$$
{\bf F}_k^{(+)}=-\nu\zeta^{''}{\bf V}_k-2\nu\zeta'\partial_z{\bf V}_k+i\omega \zeta'{\bf W}_k-\nu \partial_z^2(\zeta'{\bf W}_k)\in (L_{2,\beta_+}\cap L_{2,\beta_-})({\mathcal C})
$$
for $k=1,2$ and
$$
{\bf F}_3^{(+)}=-\nu\zeta^{''}{\bf V}_k-2\nu\zeta'\partial_z{\bf V}_k+\zeta'{\bf P}^{(+)}\in (L_{2,\beta_+}\cap L_{2,\beta_-})({\mathcal C})
$$
Similar formulas are valid for $({\bf V}^{(-)},{\bf P}^{(-)})$ with
$$
{\bf F}_k^{(-)}=-{\bf F}_k^{(+)}.
$$
By Proposition \ref{Prop1}(ii) this implies
$$
({\bf V}^{(+)},{\bf P}^{(+)})+({\bf V}^{(-)},{\bf P}^{(-)})=(0,0,p_0\hat{v}(y),p_0z+p_1)
$$
for certain constants $p_0$ and $p_1$, which furnishes the proof of the assertion.

\section{Stokes flow in a vessel with elastic walls}\label{Sect2x}

This section is devoted to the proof of Theorem \ref{T2}.
 As in the case of the Dirichlet problem considered in Sect.\ref{Sect1} we represent solutions to the problem (\ref{Intr1})--(\ref{Feb5})  in the form (\ref{Mars20a}) (for the velocity ${\bf v}$ and the pressure $p$) and
\begin{equation}\label{Mars20b}
{\bf u}(s,z,t)=\sum_{k=-\infty}^\infty e^{2\pi kit/\Lambda}{\bf U}_k(s,z)
\end{equation}
for the displacements ${\bf u}$. The coefficients in (\ref{Mars20a}) are given by (\ref{Mars20aa}) and in (\ref{Mars20b}) by
$$
{\bf U}_k(s,z)=\frac{1}{\Lambda}\int_0^\Lambda e^{-2\pi kit/\Lambda}{\bf u}(s,z,t)dt.
$$
The above introduced  coefficients satisfy the time independent problem
\begin{equation}\label{Feb5b}
i\omega{\bf V}-\nu\Delta {\bf V}+\nabla
P={\bf F}\;\;\mbox{and}\;\;\nabla\cdot{\bf V}=0\;\;\mbox{in ${\mathcal C}$},
\end{equation}
\begin{eqnarray}\label{M1ka}
&&D(\kappa(s),-\partial_s,-\partial_z)^T\,\overline{Q}(s)
D(\kappa(s),\partial_s,\partial_z){\bf
U}(s,z)-\overline{\rho}(s)\omega^2{\bf U}(s,z)\nonumber\\
&&+K{\bf U}+\sigma{\widehat{\mathcal F}}(s,z)={\bf G},
\end{eqnarray}
\begin{equation}\label{Feb5c}
{\bf V}=i\omega {\bf U}\;\;\mbox{on $\Gamma\times\Bbb R$}
\end{equation}
where ${\bf F}=0$, ${\bf G}=0$ and $\omega =2\pi k/\Lambda$ (for forthcoming analysis it is convenient to have arbitrary right-hand sides in this problem).

\begin{Theorem}\label{T26a} Let $\omega\in\Bbb R$ and $\omega\neq 0$. Then there exists $\beta>0$ depending on $\omega$ such that the only solution to the homogeneous (${\bf F}=0$ and ${\bf G}=0$) problem (\ref{Feb5b})--(\ref{Feb5c}) subject to
\begin{eqnarray}\label{Feb17aw}
&&\int_{\mathcal C}e^{-\beta|z|} (|{\bf v}|^2+|\nabla_x {\bf v}|^2+|\nabla_x\partial_z {\bf v}|^2+|p|^2)dx\\
&&+\int_{\Gamma}e^{-\beta|z|}(|u|^2+\sum_{k=2}^3(|\nabla_{sz}u_k|^2+|\nabla_{sz}\partial_zu_k|^2))dsdz< \infty\nonumber
\end{eqnarray}
 is ${\bf V}=0$, ${\bf U}=0$ and ${\bf P}=0$.
\end{Theorem}

We postpone the proof of the formulated theorem to Sect.\ref{Sect27a} and in the next section we give the proof of Theorem \ref{T2}

\subsection{Proof of Theorem \ref{T2}}\label{SEC3a}

By (\ref{Feb17aw}),
\begin{equation*}
\int_{\mathcal C}(1+|x|^2)^{-N}\big(|\nabla {\bf V}_k|^2+|P_k|^2\big)dydz+\int_\Gamma (1+|x|^2)^{-N}|{\bf U}_k|^2dsdz<\infty .
\end{equation*}
Applying Theorem \ref{T26a} we get  ${\bf V}_k=0$, ${\bf P}_k=0$ and ${\bf U}_k=0$ for $k\neq 0$.
Now using Theorem \ref{T4} and consideration in forthcoming Sect.\ref{Sub6} for $\omega=0$ we arrive at the required assertion.

\subsection{System for coefficients (\ref{Feb5b})--(\ref{Feb5c})}\label{Sub1}

To formulate the main solvability result for the system (\ref{Feb5b})--(\ref{Feb5c}),
we need the following function spaces
$$
{\mathcal Y}_\beta=\{ {\bf U}=({\bf U}_1,{\bf U}_2,{\bf U}_3):{\bf U}_1\in L_2^\beta(\Gamma),\,{\bf U}_2,{\bf U}_3\in W^{1,2}_\beta(\Gamma)\}
$$
and
$$
{\mathcal Z}_\beta=\{({\bf V},{\bf U})\,:\,{\bf V}\in W^{1,2}_\beta({\mathcal C})^3,\,{\bf U}\in {\mathcal Y}_\beta,\,i\omega{\bf U}={\bf V}\;\mbox{on}\;\Gamma\}.
$$

\begin{Proposition}\label{T29a} Let $\omega\in\Bbb R$ and $\omega\neq 0$. There exist a positive number $\beta^*$ depending on $\omega$, $\Omega$ and $\nu$ such that for any $\beta\in (-\beta^*,\beta^*)$ the following assertions hold

{\rm (i)} If ${\bf F}\in L^2_{\beta }({\mathcal C})$, ${\bf G},\partial_z{\bf G}\in L^2_\beta(\Gamma)$
then the problem (\ref{Feb5b})--(\ref{Feb5c}) has a unique solution $({\bf V},{\bf U}) \in {\mathcal Z}_\beta$, $P\in L^2_\beta({\mathcal C})$  and this solution satisfies the estimate
\begin{eqnarray*}
&&||{\bf V};W^{1,2}_\beta({\mathcal C})||+||P;L^2_\beta({\mathcal C})||+||{\bf U};{\mathcal Y}_\beta||\\
&&\leq C\Big(||{\bf F};L^2_\beta({\mathcal C})||+||{\bf G};L^2_\beta(\Gamma)||+||\partial_z{\bf G};L^2_\beta(\Gamma)||\Big),
\end{eqnarray*}
where $C$ may depend on $\omega$, $\beta$, $\nu$ and $\Omega$. Moreover,
\begin{equation*}
||\partial_z{\bf V};W^{1,2}_\beta({\mathcal C})||+||\partial_z{\bf U};{\mathcal Y}_\beta||\leq C\Big(||{\bf F};L^2_\beta({\mathcal C})||+||{\bf G};L^2_\beta(\Gamma)||+||\partial_z{\bf G};L^2_\beta(\Gamma)||\Big).
\end{equation*}

{\rm (ii)} If ${\bf F}\in L^{2}_{\beta_1 }({\mathcal C})\cap L^{2}_{\beta_2 }({\mathcal C})$, ${\bf G},\partial_z{\bf G}\in L^{2}_{\beta_1 }(\Gamma)\cap L^{2}_{\beta_2 }(\Gamma)$ with $\beta_1,\beta_2\in (-\beta^*,\beta^*)$ and $({\bf V}^{(k)},{\bf U}^{(k)},{\bf P}^{(k)})\in{\mathcal Z}_{\beta_k}\times L^2_{\beta_k}({\mathcal C})$ is the solution from {\rm (i)} for $k=1,2$ respectively, then they coincide.

\end{Proposition}

\subsection{Transformations of the problem (\ref{Intr1}), (\ref{M1k}),  (\ref{Feb5})}\label{Sub2}

It is convenient to rewrite the Stokes system (\ref{Feb5ba}) in the form
\begin{equation}\label{Feb19ab}
i\omega{\bf V}_j-\sum_{i=1}^3\partial_{x_i}T_{ji}({\bf V})={\bf F}_j,\;\; j=1,2,3,\;\nabla\cdot {\bf V}=0\;\;\mbox{in ${\mathcal C}$}.
\end{equation}
where
\begin{equation}\label{Feb19aa}
T_{ji}({\bf V})=-p\delta_{i,j}+\nu \big(\partial_{x_i}{\bf V}_j+\partial_{x_j}{\bf V}_i\big)
\end{equation}
and $\delta_{i,j}$ is the Kronecker delta. Moreover, relations (\ref{M1k}) and (\ref{Feb5}) become
\begin{eqnarray}\label{Feb21a}
&&D(\kappa(s),-\partial_s,-\partial_z)^T\,Q(s)
D(\kappa(s),\partial_s,\partial_z){\bf
U}(s,z)\nonumber\\
&&-\rho(s)\omega^2{\bf U}(s,z)+K{\bf U}+\sigma{\widehat{\mathcal F}}(s,z)={\bf G}(n,s,z)
\end{eqnarray}
and
\begin{equation}\label{Feb5cq}
{\bf V}=i\omega {\bf U}\;\;\mbox{on $\Gamma$.}
\end{equation}
Here ${\widehat{\mathcal F}}=e^{-i\omega t}{\mathcal F}$.

Next step is the application of the Fourier transform. We set
$$
{\bf V}(x)=e^{\lambda z}v(y), P(x)=e^{\lambda z}p(y)\;\;\mbox{and}\;\;{\bf U}(x)=e^{\lambda z}u(y).
$$
As the result we obtain the system
\begin{equation}\label{Feb21aa}
i\omega v_j-\sum_{i=1}^2\partial_{x_i}t_{ji}(v;\lambda)+\lambda t_{j3}(v;\lambda)=f_j\;\; j=1,2,3,
\end{equation}
\begin{equation}\label{Feb21ab}
\nabla_y\cdot v'+\lambda v_3=h\;\;\mbox{in $\Omega$,}
\end{equation}
where
\begin{equation}\label{Feb21aba}
t_{ij}(v,p;\lambda)=-p\delta_j^i+2\nu\varepsilon_{ij}(v;\lambda),
\end{equation}
\begin{eqnarray*}
&&\varepsilon_{ij}(v)=\frac{1}{2}\Big(\partial_{x_i}v_j+\partial_{x_j}v_i\Big),\;i,j\leq 2,\\
&&\varepsilon_{i3}(v;\lambda)=\hat{\varepsilon}_{3i}(v;\lambda)=\frac{1}{2}\Big(\lambda v_i+\partial_{x_i}v_3\Big),\;i=1,2,\\
&&\varepsilon_{33}(v;\lambda)=\lambda v_3.
\end{eqnarray*}
The equations (\ref{Feb21a}) and (\ref{Feb5cq}) take the form
\begin{eqnarray}\label{Feb21az}
&&D(\kappa(s),-\partial_s,-\lambda)^T\,Q(s)
D(\kappa(s),\partial_s,\lambda)u\nonumber\\
&&-\overline{\rho}(s)\omega^2u+K(s)u+\sigma(s)\Phi(s)=g(s)
\end{eqnarray}
and
\begin{equation}\label{Feb21ca}
v=i\omega u\;\;\mbox{on $\partial\Omega$.}
\end{equation}
Here $\Phi (s)=(\Phi_n,\Phi_z,\Phi_s)$ and
\begin{equation}\label{M2z}
\Phi_n=-p+2\nu\frac{\partial v_n}{\partial n},\;\;\Phi_s=\nu\Big
(\frac{\partial v_n}{\partial s}+\frac{\partial v_s}{\partial
n}-\kappa v_s\Big ),\;\;\Phi_z=\nu\Big (\lambda v_n+\frac{\partial v_z}{\partial n}\Big).
\end{equation}

\subsection{Weak formulation and function spaces}\label{Sub3}

Let us introduce an energy integral
$$
E({\bf v},\hat{\bf v})=\int_\Omega\sum_{i,j=1}^2\varepsilon_{ij}({\bf v})\overline{\varepsilon_{ij}(\hat{\bf v})}dy
$$
and put
$$
a(u,\hat{u};\lambda)=\int_{\partial\Omega} \langle Q(s)
D(\kappa(s),\partial_s,\lambda)
u(s),D(\kappa(s),\partial_s,-\bar{\lambda})\hat{u}(s)\rangle ds,
$$
where $\langle\cdot,\cdot \rangle$ is the euclidian inner product in $\Bbb C^3$. Since the matrix $Q$ is positive definite
\begin{equation}\label{Feb26aa}
a(u,u;i\xi)\geq c_1(|\xi|^2|u_3|^2+|\kappa u_1+\partial_s u_2|^2+|i\xi u_2+\partial_su_3|^2),
\end{equation}
where $\xi\in\Bbb R$ and $c_1$ is a positive constant independent of $\xi$. Another useful inequality is the following
\begin{equation}\label{Feb26aaa}
\int_{\partial\Omega}|v|^2dy\leq c_2||v||_0\,||v||_1,
\end{equation}
or by using Korn's inequality
\begin{equation}\label{Feb26aab}
q\int_{\gamma}|v|^2dy\leq c_3\Big(q^2||v||_0^2+ E(v,v)\Big)\;\;\mbox{for $q\geq 1$},
\end{equation}
 where $c_3$ does not depend on $q$.

To define  a weak  solution, we introduce  the vector function spaces:
$$
X=\{v=(v_1,v_2,v_3)\,:\,v\in H^1(\Omega)^3\; \},
$$
$$
Y=\{u_1\in H^{1/2}(\gamma)\,:\, u_2,\,u_3\in H^1(\gamma)\}
$$
and
$$
Z=Z_\omega=\{(v,u)\,:\,v\in X ,\; u\in Y,\;\;v=i\omega u\;\,\mbox{on}\;\gamma\}.
$$
We supply the space $Z$ with the inner product
\begin{eqnarray}\label{Mars24a}
&&\langle v,u;\hat{v},\hat{u}\rangle_0=\int_\Omega (v\cdot\overline{\hat{v}}+\sum_{j=1}^3\nabla_yv_j\cdot\nabla_y\overline{\hat{v}_j})dy\nonumber\\
&&+\int_\gamma (\partial_s u_2\partial_s\overline{\hat{u}_2}+\partial_s u_3\partial_s\overline{\hat{u}_3})ds.
\end{eqnarray}
Since $\omega\neq 0$, the norm $||u_1;H^{1/2}(\gamma)||$ is estimated by the norm of $v$ in the space $H^1(\gamma)$ therefore we do not need a term with $u_1$ and $\hat{u}_1$ in (\ref{Mars24a}), indeed.
Let also
$$
\langle v,p,u;\hat{v},\hat{p},\hat{u}\rangle_1=\langle v,u;\hat{v},\hat{u}\rangle_0+\int_\Omega p\overline{\hat{p}}dx
$$
be the inner product in $Z\times L^2(\Omega)$. 

We introduce also a sesqui-linear form corresponding to the formulation (\ref{Feb21aa}),
(\ref{Feb21ab}), (\ref{Feb21az}), (\ref{Feb21ca}):
\begin{eqnarray*}
&&\widehat{\mathcal A}(v,p,u;\hat{v},\hat{p},\hat{u};\lambda)=\int_{\Omega}\big(i\omega v\overline{\hat{v}}+2\nu\sum \hat{\varepsilon}_{ij}(v;\lambda)\overline{\hat{\varepsilon}_{ij}(\hat{v};-\bar{\lambda})}\big)dx+\int_\Omega p\overline{(\nabla_y\cdot \hat{v}'-\bar{\lambda}\hat{v}_3)}dx\\
&&+\int_\Omega (\nabla_y\cdot v'+\lambda v_3)\,\overline{\hat{p}}dx+i\omega(-\omega^2\int_{\partial\Omega}u\overline{\hat{u}}+a( u,\hat{u};\lambda)+k\int_\gamma u_1\overline{\hat{u}_1}ds).
\end{eqnarray*}
Clearly, this form is bounded in $Z\times L^2(\Omega)$. For ${\mathcal F}\in Z^*$, ${\mathcal H}\in L^2(\Omega)$ and $h\in L^2(\Omega)$  the weak formulation reads as the integral identity
\begin{equation}\label{24a}
{\mathcal A}(v,p,u;\hat{v},\hat{p},\hat{u};\lambda)={\mathcal F}(\hat{v},\hat{u})+\int_\Omega {\mathcal H}\overline{\hat{p}}dy
\end{equation}
which has to be valid for all $(\hat{v},\hat{p},\hat{u})\in Z\times L^2(\Omega)$ and $\nabla_y\cdot v'+\lambda v_3=h$ in $\Omega$, where $v'=(v_1,v_2)$.

If
$\nabla_y\cdot v'+\lambda v_3=0$, it is enough to require that
\begin{equation}\label{Mars19aa}
\widehat{\mathcal A}(v,p,u;\hat{v},0,\hat{u};\lambda)={\mathcal F}(\hat{v},\hat{u})
\end{equation}
for all $(\hat{v},\hat{u})\in Z$.

It will be useful to introduce the operator pencil in the space ${\mathcal Z}\times L^2(\Omega)$ depending on the parameter $\lambda\in\Bbb C$ by
\begin{equation}\label{Mars16a}
\widehat{\mathcal A}(v,p,u;\hat{v},\hat{p},\hat{u};\lambda)=\langle \Theta(\lambda)(v,u,p);\hat{v},\hat{p},\hat{u}\rangle_1.
\end{equation}
Clearly
\begin{equation}\label{Mars20ba}
\Theta(\lambda)\,:\, Z\times L^2(\Omega)\mapsto Z\times L^2(\Omega)
\end{equation}
 is a bounded operator pencil, quadratic  with respect to $\lambda$.

\subsection{Properties of the operator pencil $\Theta$}\label{Sub4}

We will need the following known lemma on the divergence equation
\begin{equation}\label{div}
\partial_{y_1}v_1+\partial_{y_2}v_2=h\;\;\mbox{in $\Omega$}.
\end{equation}
\begin{Lemma}\label{LK1}
There exists a linear operator
$$
L^2(\Omega)\ni h\rightarrow v'=(v_1,v_2)\in (H^1(\Omega))^2
$$
such that the equation {\rm (\ref{div})} is satisfied, $v_s=0$ on $\gamma$, $v'|_{\gamma}\in H^1(\gamma)$ and
\begin{equation}\label{Feb27a}
||v';H^1(\Omega)||+||v_n|_{\gamma};H^1(\gamma)||\leq C||h||_0.
\end{equation}
Clearly the vector function $(v,u)$ belongs to $Z$, where $v=(v_1,v_2,0)$ and $i\omega u=v|_\gamma$. Estimate (\ref{Feb27a}) implies
\begin{equation}\label{Feb27aw}
||(v,u);Z||\leq C||h||_0.
\end{equation}
\end{Lemma}
\begin{proof} We represent $h$ as
$$
h(y)=\check{h}(y)+\tilde{h},\;\;\tilde{h}=\frac{1}{|\Omega|}\int_\Omega h(x)dx.
$$
Then $v=\check{v}+\tilde{v}$, where $\check{v}\in H^1_0(\Omega)$ solves the problem $\nabla_y\cdot\check{v}=\check{h}$ in $\Omega$ and $\tilde{v}$ is a solution to
$$
\nabla_y\cdot\tilde{v}=\tilde{h}\;\;\mbox{and}\;\;\tilde{v}_n=\tilde{c}=\frac{|\partial\Omega|}{|\Omega|}\tilde{h}.
$$
Both mappings $\check{h}\mapsto\check{v}$ and $\tilde{h}\mapsto\tilde{v}$ can be chosen linear and satisfying
$$
||\check{v};H^1(\Omega)||\leq C||\check{h};L^2(\Omega)||\;\;\mbox{and}\;||\tilde{v};H^2(\Omega)||\leq C||\tilde{h};H^1(\Omega)||
$$
respectively. This implies the required assertion.

\end{proof}

\begin{Lemma}\label{Lka1} Let $\omega\in\Bbb R$ and $\omega\neq 0$. Then the operator pencil $\Bbb C\ni\lambda\mapsto\Theta(\lambda)$ possesses the following properties:

{\rm (i)} $\Theta(\lambda)$ is a Fredholm operator for all $\lambda\in\Bbb C$ and its spectrum consists of isolated eigenvalues
of finite algebraic multiplicity. The line $\Re\lambda=0$ is free of the eigenvalues of $\Theta$.

{\rm (ii)} Let $\lambda=i\xi$, $\xi\in \Bbb R$. Then there exists  a  positive constant $\rho(|\omega|)$ which may depend on $|\omega|$ such that the solution of  problem (\ref{Mars19aa}) with ${\mathcal F}=(f,g)\in L^2(\Omega)\times L^2(\gamma)$ and  $|\xi|\geq \rho(|\omega|)$   satisfies the estimate
\begin{eqnarray}\label{Mars24ca}
&&|\xi|^2\int_{\gamma}(|\xi|^2(|u_2|^2+|u_3|^2)+|\partial_s u_2|^2+|\partial_su_3|^2)ds+\int_\Omega |p|^2dy\nonumber\\
&&+|\xi|^2\int_\Omega (|\xi|^2|v|^2+|\nabla_yv|^2)dy\leq CN(f,g;\xi)^2,
\end{eqnarray}
where
\begin{equation}\label{Apr2a}
N(f,g;\xi)=\Big(||f||_0+||g_2||_0+||g_3||_0+|\xi|^{1/2}\,||g_1||_0\Big).
\end{equation}
The constant $c$ here may depend on $\omega$ but it is independent of $\xi$.
\end{Lemma}
\begin{proof} Let $\lambda=i\xi$ and
$$
{\mathcal X}={\mathcal X}(\lambda)=\{(v,u)\in {\mathcal Z}_\omega\,:\;\nabla\cdot v'+i\xi v_3=0\}.
$$

(i)
Consider the integral identity
\begin{equation}\label{Feb26a}
{\mathcal A}(v,0,u;\hat{v},\hat{u};i\xi)={\mathcal F}(\hat{v},\hat{u}) \;\;\mbox{ $\forall (\hat{v},\hat{u})\in{\mathcal X}$}.
\end{equation}
 We want to apply the Lax-Milgram lemma to find solution $(v,u)\in {\mathcal X}$. First, we note that
\begin{equation}\label{Mars13a}
\Re{\mathcal A}(v,0,u;v,u;i\xi)\geq c\big( I(v,v)+|\lambda|^2\int_\Omega |v_3|^2dy+\int_\Omega |\lambda v'+\nabla_yv_3|^2dy\big)
\end{equation}
and
\begin{eqnarray}\label{Mars16aa}
&&|\Im{\mathcal A}(v,0,u;v,u;i\xi)|\geq c|\omega|\Big (\int_\Omega |v|^2dy+\int_\gamma (k|u_1|^2-\omega^2|u|^2)ds\nonumber\\
&&+\int_\gamma(|\xi|^2|u_3|^2+|\kappa u_1+\partial_s u_2|^2+|i\xi u_2+\partial_su_3|^2)ds\Big),
\end{eqnarray}
where the constant $c$ does not depend on $\omega$ and $\xi$.

We use the representation
\begin{equation}\label{Mars17a}
|i\xi v+\nabla_yv_3|^2=(1-\tau)|i\xi v+\nabla_yv_3|^2+\tau|\xi|^2|v|^2+\tau|\nabla_yv_3|^2+2\tau\Re(i\xi v\cdot\nabla_y\overline{v_3}),
\end{equation}
where $\tau\in (0,1]$. Let us estimate the last term in (\ref{Mars17a}).
We have
\begin{eqnarray*}
&&\int_\Omega  v\cdot\nabla_y\overline{v_3}dx=-\int_\Omega \nabla_y\cdot v\,\overline{v_3}dx+\int_{\partial\Omega}v_n\overline{v_3}ds\\
&&=-\int_\Omega (\hat{\varepsilon}_{11}(v)+\hat{\varepsilon}_{22}(v))\overline{v_3}dx+|\omega|^2\int_{\partial\Omega}u_n\overline{u_3}ds.
\end{eqnarray*}
Since
\begin{equation}\label{Feb26b}
\Big|\xi\int_\Omega (\hat{\varepsilon}_{11}(v)+\hat{\varepsilon}_{22}(v))\overline{v_3}dx\Big|\leq \frac{1}{2}|\xi|^2\int_\Omega |v_3|^2dx+\int_\Omega( |\hat{\varepsilon}_{11}(v)|^2+|\hat{\varepsilon}_{22}(v)|^2)dy,
\end{equation}
we derive that
$$
2\tau|\Re(i\xi v\cdot\nabla_y\overline{v_3})|\leq C_\omega \tau\big(\int_\gamma (|\xi|^2|u_3|^2+|u|^2)ds+|\xi|^2\int_\Omega |v_3|^2dx+I(v,v)\big)
$$
Using above inequalities together with (\ref{Feb26aab}) for a small $\epsilon$, we arrive at the estimate
\begin{eqnarray}\label{Mars24b}
&&|{\mathcal A}(v,0,u;v,u;i\xi)|\geq C_\omega\Big(\int_\Omega (|\xi|^2|v|^2+|\nabla_yv|^2)dy\nonumber\\
&&+\int_{\gamma}(|\xi|^2|u_3|^2+|\partial_s u_2|^2+|i\xi u_2+\partial_su_3|^2)ds\Big),
\end{eqnarray}
where $C_\omega$ is a positive constant which may depend on $\omega$ and $|\xi|$ is chosen to be sufficiently large with respect to $|\omega|+1$. On the basis of
$$
\int_{\gamma}|i\xi u_2+\partial_su_3|^2ds=\int_{\gamma}\big(|\xi u_2|^2+|\partial_su_3|^2- 2\Re(i\xi \partial_su_2\bar{u}_3)\big)ds
$$
one can continue the estimation in (\ref{Mars24b}) as follows:
\begin{eqnarray}\label{Mars24ba}
&&|{\mathcal A}(v,0,u;v,u;i\xi)|\geq C_\omega\Big(\int_\Omega (|\xi|^2|v|^2+|\nabla_yv|^2)dy\nonumber\\
&&+\int_{\gamma}(|\xi|^2|u_3|^2+|\partial_s u_2|^2+\xi^2 u_2|^2+|\partial_su_3|^2)ds\Big),
\end{eqnarray}
with possibly another constant $C_\omega$.
Application of the Lax-Milgram lemma gives existence of a unique solution in ${\mathcal X}$ and the following estimate for this solution
\begin{eqnarray}\label{Mars23a}
&&\int_\Omega (|\xi|^2|v|^2+|\nabla_yv|^2)dy\\
&&+\int_{\partial\Omega}(|\xi|^2|u_3|^2+|\partial_s u_2|^2+\xi^2 |u_2|^2+|\partial_su_3|^2)ds\leq C||{\mathcal F};{\mathcal Z}^*||^2\nonumber
\end{eqnarray}
with a constant $C$ which may depend on $\omega$ and $\xi$.
It remains to estimate the function $p$. We chose the test function $(\hat{v},\hat{u})$ in the following way: $i\omega\hat{u}=v$ on $\partial\Omega$, $\hat{v_3}=0$ and  $\hat{v}'\in H^1(\Omega)$ solves the problem
$$
\nabla_y\cdot\hat{v}'=\hat{h}\;\;\mbox{in $\Omega$}
$$
where $h\in L^2(\Omega)$. According to Lemma \ref{LK1} the mapping $h\mapsto\hat{v}'$ can be chosen linear and satisfying the estimate (\ref{Feb27a}). The pressure $p$ must satisfy the relation
\begin{eqnarray}\label{Mars23aa}
&&\int_\Omega p\overline{\hat{h}}dy={\mathcal F}(\hat{v},\hat{u})-\int_{\Omega}\big(i\omega v\overline{\hat{v}}+2\nu\sum \varepsilon_{ij}(v)\overline{\varepsilon_{ij}(\hat{v})}\big)dx\nonumber\\
&&-i\omega(\omega^2\int_{\partial\Omega}u\overline{\hat{u}}+a( u,\hat{u};\lambda)).
\end{eqnarray}
One can verify using (\ref{Mars23a}) that the right-hand side of (\ref{Mars23aa}) is a linear bounded functional with respect to $h\in L^2(\Omega)$ and therefore there exists $p\in L^2(\Omega)$ solving (\ref{Mars23aa}) and estimated by the corresponding norm of ${\mathcal F}$. Thus the operator pencil (\ref{Mars20ba}) is isomorphism for large $|\xi|$.

Since the operator $\Theta(\lambda_1)-\Theta(\lambda_2)$ is compact we obtain that the spectrum of the operator pencil $\Theta$ consists of isolated eigenvalues of finite algebraic multiplicities, see \cite{GK}.

Let us show  that the kernel of $\Theta(i\xi)$ is trivial for all $\xi\in\Bbb R$. Indeed, if $(v,u,p)\in {\mathcal Z}_\omega\times L_2(\Omega)$ is a weak solution with ${\mathcal F}=0$ then in the case $\xi\neq 0$ inequality (\ref{Mars13a}) implies $v=0$ and hence $p=0$ because $i\omega u=v$ on $\gamma$. From (\ref{Mars16a}) it follows that
$$
\int_\Omega p\overline{(\nabla_y\cdot\hat{v}'+i\xi\hat{v}_3)}dy=0.
$$
By Lemma \ref{LK1} there exists the element $(v_1,v_2,0)\in{\mathcal X}$ solving $\nabla_y\cdot\hat{v}'=p$ which gives $p=0$. In the case $\lambda=0$ we derive from (\ref{Mars13a}) $v_3=c_3$ and $(v_1,v_2)=(c_1,c_2)+c_0(y,-x)$, where $c_0,\ldots,c_3$ are constant. From (\ref{Feb7b}) it follows that
$$
v_3=0\;\,\mbox{and}\;\;i\omega v_j+\partial_{y_j}p=0,\;j=1,2.
$$
Since the vector $v$ is a rigid displacement, we have $Du=0$ due to (\ref{Mars21a}) and (\ref{Feb21ca}). Hence relation (\ref{Feb21az}) implies
$$
-\rho(s)\omega^2u(s)+Ku(s)=\sigma (p,0,0)^T.
$$
Therefore, $u_2=0$ and $-\rho(s)\omega^2u_1+ku_1=p$. By (\ref{Feb21ca}), we have
$$
c_1\zeta_1'+c_2\zeta_2'+c_0(\zeta_2\zeta_1'-\zeta_1\zeta_2')=0.
$$
In view of Lemma \ref{La1} this yields $c_0=c_1=c_2=0$. Thus, the assertion (i) is proved

(ii) Let
$$
{\mathcal F}(\hat{v},\hat{u})=\int_\Omega f\cdot\overline{\hat{v}}dy+\int_\gamma g\cdot\overline{\hat{u}}ds,\;\;f\in L_2(\Omega,\;\;g\in L_2(\gamma).
$$
Then
\begin{eqnarray*}
&&|{\mathcal F}(v,u)|\leq \Big(||f||_0+||g_2||_0+||g_3||_0+|\xi|^{1/2}\,||g_1||_0\Big)\\
&&\times\Big(||v||_0+||u_2||_0+||u_3||_0+|\xi|^{-1/2}||u_1||_0\Big).
\end{eqnarray*}
Using (\ref{Feb26aab}) and (\ref{Mars23a}9, we get
$$
|\xi|^2\big(\int_\Omega |v|^2dy+\int_\gamma (|u_2|^2+|u_3|^2+|\xi|^{-1}|u_1|^2)ds\big)\leq C|{\mathcal F}(v,u)|.
$$
Therefore,
\begin{equation}\label{Mars24c}
|\xi|^4\big(\int_\Omega |v|^2dy+\int_\gamma (|u_2|^2+|u_3|^2+|\xi|^{-1}|u_1|^2)ds\big)\leq C\Big(||f||_0^2+||g_2||_0^2+||g_3||_0^2+|\xi|\,||g_1||_0^2\Big).
\end{equation}
Furthermore,
\begin{eqnarray*}
&&\int_\Omega |\nabla_yv|^2)dy+\int_{\gamma}(|\partial_s u_2|^2+|\partial_su_3|^2)ds\leq C|{\mathcal F}(v,u)|\\
&&\leq C\Big(||f||_0+||g_2||_0+||g_3||_0+|\xi|^{1/2}\,||g_1||_0\Big)\\
&&\times\Big(||v||_0+||u_2||_0+||u_3||_0+|\xi|^{-1/2}||u_1||_0\Big)\\
&&\leq C|\xi|^{-2}\Big(||f||_0+||g_2||_0+||g_3||_0+|\xi|^{1/2}\,||g_1||_0\Big)^2
\end{eqnarray*}
where we have used (\ref{Mars24c}). The last inequality together with (\ref{Mars24c}) delivers (\ref{Mars24ca}) for the vector functions $v$ and $u$.

To obtain the estimate for $p$, we proceed as in (i).

\end{proof}

\subsection{Solvability of the problem (\ref{Feb19ab})--(\ref{Feb5cq})}\label{Sub5}

\begin{Proposition}\label{Pr29a} Let $\omega\in\Bbb R$ and $\omega\neq 0$. There exist a positive number $\beta^*$ depending on $\omega$, $\Omega$ and $\nu$ such that  the following assertions hold:

{\rm (i)} The strip $|\Re\lambda| <\beta^*$ is free of eigenvalues of the operator pencil $\Theta$.

{\rm (ii)} For $\Re\lambda=\beta\in (-\beta^*,\beta^*)$ the  estimate
\begin{eqnarray}\label{Mars24caa}
&&(|\lambda|^2+1)\int_{\gamma}((|\lambda|^2+1)(|u_2|^2+|u_3|^2)+|\partial_s u_2|^2+|\partial_su_3|^2)ds\\
&&+\int_\Omega |p|^2dy+(|\lambda|^2+1)\int_\Omega ((|\lambda|^2+1)|v|^2+|\nabla_yv|^2)dy\leq CN(f,g;\xi)^2\nonumber
\end{eqnarray}
is valid, where $N$ is given by (\ref{Apr2a}). The positive constant $C$ here may depend on $\beta$, $\omega$, $\nu$ and $\Omega$.
\end{Proposition}
\begin{proof} Let $\lambda=\beta+i\xi$. 
It is straightforward to verify that
\begin{equation*}
\Theta (\lambda)(v,u,p)-\Theta (i\xi)(v,u,p)=\beta(A,B,C)^T
\end{equation*}
where
$$
A=-\nu(\beta+2i\xi)v+(0,0,p);\;\;B=v_3;
$$
\begin{equation*}
C=\frac{\nu\sigma}{i\omega}(u_1,0,0)-(\beta+2i\xi)D_0^TQD_0u+(D_1(-\partial_s)^TQD_0+D_0^TQD(\partial_s))u.
\end{equation*}
Therefore,
$$
\Theta (\lambda)-\Theta (i\xi):Z \times L_2(\Omega)\to L_2(\Omega)^3\times L_2(\gamma)^3\times L_2(\Omega)
$$
is a small operator for small $\beta$. Therefore the estimate (\ref{Mars24caa}) for large $|\xi|$ follows from (\ref{Mars24ca}).
From Lemma \ref{Lka1}(i) it follows that this can be extended on $\xi\in {\Bbb R}$ if $\beta$ is chosen sufficiently small. Thus we arrive at both assertions of the lemma.

\end{proof}

\subsection{Proof of Proposition \ref{T29a} and Theorem \ref{T26a}}\label{Sect27a}

\begin{proof}
 The assertion (i) in Proposition \ref{T29a} is obtained from Proposition \ref{Pr29a}(ii)  by using the inverse Fourier transform together with Parseval's identity.

To conclude with (ii) we observe that the same proposition provides
$$
({\bf V}_\pm,{\bf P}_\pm)=\frac{1}{2\pi i}\int_{\Re\lambda=\pm\beta}e^{-\lambda z}{\Theta}^{-1}(\lambda)(f(y,\lambda),g(s,\lambda),0)d\lambda
$$
and the assertion (ii) in Proposition \ref{T29a} is obtained by applying the residue theorem.
\end{proof}

\bigskip
{\bf Proof of Theorem  \ref{T26a}}.
Let $({\bf V},{\bf U}, {\bf P})$ be a solution to (\ref{Feb5b})--(\ref{Feb5c}) satisfying (\ref{Feb17aw}). Our first step is to construct a  representation of the solution in the form
\begin{equation}\label{Apr31abz}
({\bf V},{\bf U},{\bf P})=({\bf V}^{(+)},{\bf U}^{(+)},{\bf P}^{(+)})+({\bf V}^{(-)},{\bf U}^{(-)},{\bf P}^{(-)}),
\end{equation}
where $({\bf V}^{(\pm)},{\bf U}^{(\pm)})\in {\mathcal Z}_{\pm\beta}({\mathcal C})$, ${\bf P}^{(\pm)})\in L^2_{\pm\beta}({\mathcal C})$ and they solve the problem (\ref{Feb5b})--(\ref{Feb5c}) with certain $({\bf F},{\bf G})=({\bf F}^{(\pm)},{\bf G}^{(\pm)})$ such that ${\bf F}^{(\pm)}\in L^2_{\pm\beta}({\mathcal C})$ and ${\bf G}^{(\pm)},\partial_z{\bf G}^{(\pm)}\in L^2_{\pm\beta}(\Gamma)$.

Let $\zeta$ be the same  cut-off function as in the proof of Theorem \ref{T4}. 
We choose in (\ref{Apr31abz})
\begin{eqnarray}\label{Apr4b}
&&({\bf V}^{(+)},{\bf U}^{(+)},{\bf P}^{(+)})=\zeta ({\bf V},{\bf U},{\bf P})+\zeta'(\tilde{\bf V},\tilde{\bf U},\tilde{\bf P}),\nonumber\\
&&({\bf V}^{(-)},{\bf U}^{(-)},{\bf P}^{(-)})=(1-\zeta )({\bf V},{\bf U},{\bf P})-\zeta'(\tilde{\bf V},\tilde{\bf U},\tilde{\bf P}),
\end{eqnarray}
where the vector function $(\tilde{\bf V},\tilde{\bf U},\tilde{\bf P})$ solves the problem (\ref{24a}) for $\lambda=0$ and with ${\mathcal F}=0$, ${\mathcal H}=0$ and $h={\bf V}_3(y,z)$ i.e.
\begin{equation}\label{Apr4a}
\partial_{y_1}{\bf W}_1+\partial_{y_2}{\bf W}_2={\bf V}_3(y,z)\;\;\mbox{for $y\in\Omega$}.
\end{equation}
In this problem the variable $z$ is considered as a parameter. In order to apply Lemma \ref{Lka1}(i) we reduce the above formulation to the case $h=0$. Applying for this purpose Lemma \ref{LK1} we find a function ${\mathcal V}(y)=({\mathcal V}_1,{\mathcal V}_2,{\mathcal V}_3)(y)$ solving (\ref{Apr4a}) and satisfying (\ref{Feb27a}) or (\ref{Feb27aw}) where $i\omega {\mathcal U}={\mathcal V}$. The function $({\mathcal V},{\mathcal U})\in Z$ and the above formulation is reduced to the case $h=0$ but with some nonzero ${\mathcal F}$. Applying to the new problem Lemma \ref{Lka1}(i) we find solution satisfying
$$
||(\tilde{\bf V},\tilde{\bf U},\tilde{\bf P});Z\times L^2(\Omega)||\leq C||{\bf V}_3||_0,
$$
which depends on the parameter $z$. Since
$$
{\bf V}_2,\;\partial_z{\bf V}_3,\partial_z^2{\bf V}_3\in L^2_{\rm loc}(\Bbb R;L_2(\Omega)),
$$
we get that
$$
(\tilde{\bf V},\tilde{\bf U},\tilde{\bf P}),\partial_z(\tilde{\bf V},\tilde{\bf U},\tilde{\bf P})\in L^2_{\rm loc}(\Bbb R;Z\times L_2(\Omega))\;\;\mbox{and}\;\partial_z^2(\tilde{\bf V},\tilde{\bf U})\in L^2_{\rm loc}(\Bbb R;L_2(\Omega)\times L_2(\gamma)).
$$
Now one can verify that the vector functions (\ref{Apr4b}) satify (\ref{Feb5b})--(\ref{Feb5c}) with certain right-hand sides $({\bf F}^{(\pm)}, {\bf G}^{(\pm)})$ having compact supports. Moreover
${\bf F}^{(+)}=-{\bf F}^{(-)}$ and ${\bf G}^{(+)}=-{\bf G}^{(-)}$ and
$$
{\bf F}^{(\pm)}\in L^2_{\pm\beta}({\mathcal C}),\;\;{\bf G}^{(\pm)},\partial_z{\bf G}^{(\pm)}\in L^2_{\pm\beta}(\Gamma).
$$
Applying Theorem \ref{T29a}(ii) we get
$$
({\bf V}^{(+)},{\bf U}^{(+)},{\bf P}^{(+)})=-({\bf V}^{(-)},{\bf U}^{(-)},{\bf P}^{(-)}),
$$
which means that $({\bf V},{\bf U},{\bf P})=0$. Theorem \ref{T26a} is proved.

\subsection{The case $\omega=0$, homogeneous system}\label{Sub6}

If $\omega=0$ the system (\ref{Intr1}))--(\ref{Feb5}) becomes
\begin{equation}\label{Feb28a}
-\nu\Delta {\bf v}+\nabla p=0\;\;\mbox{and}\;\;\nabla\cdot{\bf v}=0\;\;\mbox{in ${\mathcal C}\times\Bbb R$} .
\end{equation}
\begin{equation}\label{Feb28ab}
v=0\;\;\mbox{on $\Gamma$}
\end{equation}
and
\begin{eqnarray}\label{Feb28ac}
&&D(\kappa(s),-\partial_s,-\partial_z)^T\,\overline{Q}(s)
D(\kappa(s),\partial_s,\partial_z){\bf
u}(s,z)\nonumber\\
&&+k(s){\bf u}_1(s,z)+\sigma(s){\mathcal F}(s,z)=0,
\end{eqnarray}
where ${\mathcal F}$ is given by (\ref{M2}). So we see that the system becomes uncoupled with respect to $({\bf v},p)$ and $u$. Solutions to (\ref{Feb28a})-(\ref{Feb28ab}) are
given by
$$
{\bf v}_1={\bf v}_2=0,\;\;p=p_0z+p_1\;\;\mbox{and}\;\;{\bf v}_3=p_0{\bf v}_*(y),
$$
where $p_0$, $p_1$ are constants and ${\bf v}_*$ solves the problem (\ref{M9a}). In this case the vector function ${\mathcal F}$ is evaluated as
$$
{\mathcal F}_1={\mathcal F}_n=-(p_0z+p_1),\;\;{\mathcal F}_2={\mathcal F}_s=0\;\;\mbox{and}\;\;{\mathcal F}_3={\mathcal F}_z=\nu\partial_n{\bf v}_3.
$$
Let
\begin{equation}\label{Feb28b}
Q=\left(\begin{array}{lll}
Q_{11} & Q_{12} & Q_{13}  \\
Q_{21} &Q_{22} &Q_{23}\\
Q_{31} & Q_{32} &Q_{33}
\end{array}
\right )
\end{equation}

First we consider the case $p_0=p_1=0$. Namely, we want to solve the homogeneous equation
\begin{equation}\label{Feb28ba}
(-D_0\partial_z+D_1(-\partial_s))^TQ(s)
(D_0\partial_z+D_1(\partial_s)){\bf U}+k{\bf U}_1(s)=0,
\end{equation}
where $D_0$ and $D_1$ are defined by (\ref{Intr6j}). First, we are looking for solution independent of $z$. Then it
 must satisfy
$$
\kappa {\bf U}_1+\partial_s{\bf U}_2=0,\;\;\partial_s{\bf U}_3=0\;\;\mbox{and}\;\;{\bf U}_1=0.
$$
Thus
$$
{\bf U}_1=0,\;\;{\bf U}_2=c_2,\;\;{\bf U}_3=c_3,
$$
where $c_2$ and $c_3$ are constants.

Next let ${\bf U}$ be a linear function in $z$, i.e. ${\bf U}(s,z)=z{\bf u}^0(s)+{\bf u}^1(s)$. Then
$$
D_1(-\partial_s)^TQ(s)
D_1(\partial_s){\bf u}^0+K{\bf u}^0=0,\;\;
$$
and
\begin{equation}\label{M4b}
D_1(-\partial_s)^TQ(s)
D_1(\partial_s){\bf u}^1+K{\bf u}^1+\Big(D_1(-\partial_s)^TQ(s)
D_0-D_0^TQ(s)
D_1(\partial_s)\Big){\bf u}^0=0.
\end{equation}
Since ${\bf u}^0=(0,\alpha,\beta)^T$, $\alpha$ and $\beta$ are constant,  equation (\ref{M4b}) takes the form
\begin{equation}\label{M4bz}
D_1(-\partial_s)^TQ(s)
(D_1(\partial_s){\bf u}^1+D_0{\bf u}^0)+K{\bf u}^1=0
\end{equation}
and it is solvable since the term containing ${\bf u}^0$ is orthogonal to constant vectors $(0,a_1,a_2)$.
Thus there exists linear in $z$ solutions. Let us find these solutions. We have
\begin{equation}\label{M6bb}
Q(D_1(\partial_s){\bf u}^1+D_0{\bf u}^0)=(A,B,C)^T,
\end{equation}
where
\begin{eqnarray*}
&&A=Q_{11}(\kappa {\bf u}^1_1+\partial_s{\bf u}^1_2)+Q_{13}\frac{1}{\sqrt{2}}\partial_s{\bf u}^1_3+Q_{12}\beta+\frac{1}{\sqrt{2}}Q_{13}\alpha,\\
&&B=Q_{21}(\kappa {\bf u}^1_1+\partial_s{\bf u}^1_2)+Q_{23}\frac{1}{\sqrt{2}}\partial_s{\bf u}^1_3+Q_{22}\beta+\frac{1}{\sqrt{2}}Q_{23}\alpha,\\
&&C=Q_{31}(\kappa {\bf u}^1_1+\partial_s{\bf u}^1_2)+Q_{33}\frac{1}{\sqrt{2}}\partial_s{\bf u}^1_3+Q_{32}\beta+\frac{1}{\sqrt{2}}Q_{33}\alpha.
\end{eqnarray*}
Now system (\ref{M4bz}) takes the form
\begin{equation*}
\kappa A+k{\bf u}^1_1=0,\;\;\partial_sA=0,\;\;\partial_sC=0.
\end{equation*}
This implies
$$
A=b_1,\;\;C=b_2,\;\;\kappa b_1+k{\bf u}^1_1=0,
$$
where $b_1$ and $b_2$ are constants. Therefore
\begin{equation}\label{M6a}
{\bf u}^1_1=-\frac{\kappa b_1}{k}
\end{equation}
and
\begin{equation}\label{M6aa}
(\kappa {\bf u}^1_1+\partial_s{\bf u}^1_2,\frac{1}{\sqrt{2}}\partial_s{\bf u}^1_3)^T={\mathcal R}(s)(b_1,b_2)^T-{\mathcal R}(s){\mathcal S}(\beta,\frac{1}{\sqrt{2}}\alpha)^T,
\end{equation}
where
\begin{equation*}
{\mathcal R}(s)=\left(\begin{matrix}Q_{11}&Q_{13}\\
Q_{31}&Q_{33}\\
\end{matrix}\right)^{-1},\;\;\;{\mathcal S}=\left(\begin{matrix}Q_{12}&Q_{13}\\
Q_{32}&Q_{33}\\
\end{matrix}\right).
\end{equation*}
Using (\ref{M6a}) we can write the compatibility condition for (\ref{M6aa}) as
\begin{equation}\label{M6bbz}
\int_0^{|\gamma|}{\mathcal R}(s)ds(b_1,b_2)^T+\int_0^{|\gamma|}\frac{\kappa^2}{k}ds(b_1,0)^T=\int_0^{|\gamma|}{\mathcal R}(s){\mathcal S}ds(\beta,\frac{1}{\sqrt{2}}\alpha)^T.
\end{equation}
Since ${\mathcal R}$ is a positive definite matrix, this system is uniquely solvable with respect to $(b_1,b_2)$.

Next let us look for solution to (\ref{Feb28ba}) in the form
$$
{\bf U}=\frac{1}{2}z^2{\bf u}^0+z{\bf u}^1+{\bf u}^2,
$$
where ${\bf u}^0$ and ${\bf u}^1$ just constructed above vector functions. Then equation for ${\bf u}^2$ has the form
\begin{equation}\label{M6b}
D_1(-\partial_s)^TQ(s)(
D_1(\partial_s){\bf u}^2+D_0{\bf u}^1)+K{\bf u}^2-D_0^TQ(s)(D_1(\partial_s){\bf u}^1+QD_0{\bf u}^0)=0.
\end{equation}
According to (\ref{M6bbz}) solvability of this system is equivalent to
$$
\int_0^{|\gamma|}Bds=0\;\;\mbox{and}\;\;\int_0^{|\gamma|}Cds=0.
$$
this means that $b_2=0$ and
\begin{eqnarray}\label{M6c}
&&\int_0^{|\gamma|}\Big((Q_{22},Q_{23})-(Q_{21},Q_{23}){\mathcal R}{\mathcal S}\Big)ds(\beta,\frac{1}{\sqrt{2}}\alpha)^T\nonumber\\
&&+\int_0^{|\gamma|}(Q_{21},Q_{23}){\mathcal R}ds(b_1,b_2)^T=0.
\end{eqnarray}
Furthermore, $(b_2,b_1)$ and $(\alpha,\beta)$ are connected by (\ref{M6bb}). To simplify calculation we assume from now that
$$
Q_{13}=Q_{23}=0.
$$
Then the matrixes ${\mathcal R}$ and ${\mathcal S}$ are diagonal and from (\ref{M6bbz}) it follows that $b_2=0$ implies
$\alpha=0$ and
\begin{equation}\label{M6bbb}
\int_0^{|\gamma|}Q_{11}^{-1}dsb_1+\int_0^{|\gamma|}\frac{\kappa^2}{k}dsb_1=\int_0^{|\gamma|}Q_{11}^{-1}Q_{12}ds\beta.
\end{equation}
The relation (\ref{M6c}) implies
\begin{equation*}
\int_0^{|\gamma|}\Big(Q_{22}-Q_{21}Q_{11}^{-1})\Big)ds\beta
+\int_0^{|\gamma|}Q_{21}Q_{11}^{-1}dsb_1=0.
\end{equation*}
This relation together with (\ref{M6bbb}) requires that $\beta=0$ and $b_1=0$.

If $\xi\neq 0$ then ${\bf U}=0$. Consider the case $\xi=0$ and let $p_0=0$. Then
(\ref{Feb28ac}) takes the form
\begin{equation}\label{Feb28bb}
\left(\begin{array}{lll}
\kappa & 0 & 0 \\
-\partial_s &0 &0\\
 0 & 0 &-\frac{1}{\sqrt{2}}\partial_s
\end{array}
\right )
Q
\left(\begin{array}{lll}
\kappa {\bf U}_1+\partial_s{\bf U}_2 \\
0\\
\frac{1}{\sqrt{2}}\partial_s{\bf U}_3
\end{array}
\right )+K{\bf U}=-p_1\sigma
\left(\begin{array}{lll}
1 \\
0\\
0
\end{array}
\right )
\end{equation}
This is equivalent to the following three equations
$$
\kappa\Big(Q_{11}(\kappa {\bf U}_1+\partial_s{\bf U}_2)+Q_{13}\frac{1}{\sqrt{2}}\partial_s{\bf U}_3\Big)+k{\bf U}_1=-p_1\sigma,
$$
$$
-\partial_s\Big(Q_{11}(\kappa {\bf U}_1+\partial_s{\bf U}_2)+Q_{13}\frac{1}{\sqrt{2}}\partial_s{\bf U}_3\Big)=0,
$$
$$
\partial_s\Big(Q_{31}(\kappa {\bf U}_1+\partial_s{\bf U}_2)+Q_{33}\frac{1}{\sqrt{2}}\partial_s{\bf U}_3\Big)=0
$$
This implies
\begin{equation}\label{M4a}
Q_{11}(\kappa {\bf U}_1+\partial_s{\bf U}_2)+Q_{13}\frac{1}{\sqrt{2}}\partial_s{\bf U}_3=b_1,\;\;Q_{31}(\kappa {\bf U}_1+\partial_s{\bf U}_2)+Q_{33}\frac{1}{\sqrt{2}}\partial_s{\bf U}_3=b_2,
\end{equation}
where $b_1$ and $b_2$ are constants, and
$$
\kappa b_1+k{\bf U}_1=-p_1\sigma.
$$
Hence
$$
{\bf U}_1=-\frac{p_1\sigma+\kappa b_1}{k}.
$$

Solving the system
$$
Q_{11}B_1+Q_{13}B_2=b_1,\;\;Q_{31}B_1+Q_{33}B_2=b_2,
$$
we get
$$
(B_1,B_2)^T={\mathcal R}(b_1,b_2)^T,\;\;{\mathcal R}=\left(\begin{matrix}Q_{11}&Q_{13}\\
Q_{31}&Q_{33}\\
\end{matrix}\right)^{-1}.
$$
We write the equations (\ref{M4a}) as
\begin{equation}\label{Apr6a}
\partial_s{\bf U}_2=B_1-\kappa{\bf U}_1,\;\;\partial_s{\bf U}_3=\sqrt{2}B_2.
\end{equation}
These equations has periodic solutions if
$$
\int_0^{|\gamma|}(B_1-\kappa{\bf U}_1)ds=0, \int_0^{|\gamma|}B_2ds=0.
$$
We write these equations as a system with respect to $b_1$ and $b_2$
$$
|\gamma|{\mathcal R}(b_1,b_2)^T+\int_0^{|\gamma|}\frac{\kappa^2}{k}ds(b_1,0)^T
=-\Big(\int_0^{|\gamma|}\kappa\frac{p_1\sigma}{k}ds,0\Big)^T
$$
From these relations we can find $b_1$ and $b_2$ and then solving (\ref{Apr6a}) we can find ${\bf U}_2$ and  ${\bf U}_3$.

\bigskip

\section{Appendix}

\begin{Lemma}\label{La1} Let
$$
c_1\zeta_1'(s)+c_2\zeta_2'(s)+c_0(\zeta_2(s)\zeta_1'(s)-\zeta_1(s)\zeta_2'(s))=0\;\;\mbox{for $s\in (0,|\gamma|]$},
$$
where $c_0$, $c_1$ and $c_2$ are complex constants. Then $c_0=c_1=c_2=0$.
\end{Lemma}
\begin{proof} It is sufficient to prove assertion for real $c_0$, $c_1$ and $c_2$. Assume that $c_0=0$. Then
$c_1\zeta_1(s)+c_2\zeta_2(s)=c$ and hence the boundary $\gamma$ belongs to the line or $c_1=c_2=0$. Since the first option is impossible we obtain that both constants are zero.

So it is sufficient to prove that $c_0=0$. Assume that it is not. Moving the origin in the $(y_1,y_2)$ plane (replacing $\zeta_k$ by $\zeta_k+\alpha_k$, $k=1,2$), we arrive at the equation
$$
(\zeta_2(s)\zeta_1'(s)-\zeta_1(s)\zeta_2'(s))=0\;\;\mbox{for $s\in (0,|\gamma|]$.}
$$
The last relation means that at each point $(y_1,y_2)$ on $\gamma$ the corresponding vector is orthogonal to the normal to this curve at the same point, what is impossible.
\end{proof}

\bigskip
\noindent {\bf Acknowledgements.} V.Kozlov was supported by the Swedish Research Council (VR), 2017-03837.
S.Nazarov is supported by RFBR grant 18-01-00325.
This study was supported by Link\"oping University, and by RFBR grant 16-31-60112.

\end{document}